\def\phi{\varphi}
\def\dot{,\ldots,}

\def\G{\mathcal{G}}\def\P{\mathcal{P}}
\def\Voc{\mbox{Voc}}

\def\So{\mbox{SO}}
\def\Strt{\mbox{Sut}}
\def\Str{\mbox{Str}}
\def\Sat{\mbox{Sat}}
\def\soc{second order characterizable}

\def\Exists{\tilde{\exists}}
\def\Forall{\tilde{\forall}}

\def\dom{\mathop{\rm dom}}

\def\Str{{\rm Str}}

\def\ma{{\mathcal A}}

\def\mm{{\mathcal M}}
\def\mn{{\mathcal N}}

\newcommand{\open}{\Bbb}

\newcommand{\oN}{{\open N}}

\documentclass{article}
\usepackage{amsthm} 
\usepackage{graphicx} 
\usepackage{amssymb,latexsym,times,amsmath,bm}
\newcommand{\psfrag}[2]{}
\usepackage{makeidx,proof,color}
\usepackage{verbatim,hyperref,lineno,comment}
\usepackage[utf8]{inputenc}
 \theoremstyle{plain}
  \newtheorem{theorem}{Theorem}[section]
  \newtheorem{lemma}[theorem]{Lemma}
  \newtheorem{corollary}[theorem]{Corollary}

  \theoremstyle{definition}
  \newtheorem{definition}[theorem]{Definition}

  \theoremstyle{remark}

\makeindex

\makeatletter
\newcommand\chaptercontents{
{\global
\@topnum\z@ 
\@afterindentfalse 
\if@twocolumn
\@restonecoltrue
\onecolumn 
\else 
\@restonecolfalse 
\fi 
\vspace*{10pt}
\noindent 
{\small\bf Contents}\par 
\vskip1em 
\nobreak} 
{\small
\@starttoc{toc}%
}\if@restonecol
\twocolumn
\fi}
\renewcommand*\l@section[2]{%
\ifnum \c@tocdepth >\z@
\addpenalty\@secpenalty 
\setlength\@tempdima{2.5em}%
\begingroup 
\parindent \z@ 
\rightskip
\@pnumwidth \parfillskip -\@pnumwidth 
\leavevmode 
\advance\leftskip\@tempdima 
\hskip -\leftskip 
sharp1\nobreak\leaderfill\nobreak
\hb@xt@\@pnumwidth{\hss sharp2}\par 
\endgroup 
\fi}
\renewcommand*\l@section{\@dottedtocline{1}{0.1em}{1.3em}}       
\renewcommand*\l@subsection{\@dottedtocline{2}{1.5em}{2em}}      
\renewcommand*\l@subsubsection{\@dottedtocline{3}{3.5em}{2.6em}} 

\makeatother

\begin{document}

\title{SORT LOGIC AND FOUNDATIONS OF MATHEMATICS\thanks{Published in \cite{MR3205075}.}}

\author{Jouko V\"a\"an\"anen\\
Department of Mathematics and Statistics, University of Helsinki\\
Institute for Logic, Language and Computation, University of Amsterdam\\
jouko.vaananen@helsinki.fi}

\maketitle

\begin{abstract}
I have argued elsewhere \cite{meBSL} that second order logic provides a foundation for mathematics much in the same way as set theory does, despite the  fact that the former is second order and the latter first order, but second order logic is marred by reliance on ad hoc {\em large domain assumptions}. In this paper I argue that sort logic, a powerful extension of second order logic, provides a foundation for mathematics without any ad hoc large domain assumptions. The large domain assumptions are replaced by ZFC-like axioms. Despite this resemblance to set theory sort logic retains the structuralist approach to mathematics characteristic of second order logic. As a model-theoretic logic sort logic is the strongest logic. In fact, every model class definable in set theory is the class of models of a sentence of sort logic. Because of its strength sort logic can be used to formulate particularly strong reflection principles in set theory.

\end{abstract}

\vspace*{12pt}   


\section{Introduction}

Sort logic, introduced in \cite{MR567682}, is a many-sorted  extension of second order logic.  In an exact sense it is the strongest logic that there is.  In this paper sort logic is suggested as a foundation of mathematics and contrasted to second order logic and to set theory. It is argued that sort logic solves the problem of second order logic that existence proofs of structures rely on ad hoc large domain assumptions.

The new feature in sort logic over and above what first and second order logics have is the ability to ``look outside" the model, as for a group to be the multiplicative group of a field requires reference to a zero element outside the group, or for a Turing machine, defined as a finite set of quadruples, to halt requires reference to a tape potentially much bigger than the Turing machine itself. 

In computer science it is commonplace to regard a database as a many-sorted structure. Each column (attribute) of the database has its own range of values, be it a salary figure, gender, department, last name, zip code, or whatever. In fact, it would seem very unnatural to lump all these together into one domain which has a mixture of numbers, words, and strings of symbols. To state that a new column can be added to a database, e.g. a salary column, involves stating that new elements, namely the salary values, can be added to the overall set of objects referred to in the database.

In a sense ordinary second order logic  also ``looks outside" the model as well as one can think of the bound second order variables as first order variables ranging over the domain of all subsets and relations on the original domain. In fact, one of the best ways to understand second order logic is to think of it as a two-sorted first order logic in which one sort---the sort over which the second order variables range---is assumed to consist of {\em all} subsets and relations of the other sort. When ``all subsets and relations" is replaced by ``enough subsets and relations to satisfy the Comprehension Axioms", we get semantics relative to which there is a Completeness  of Henkin \cite{MR0036188}. The same is true of sort logic.

To get a feeling of sort logic, let us consider the following formulation of the field axioms in a many-sorted first order logic with two sorts of variables. We use variables $x,y$ and $z$ for the  sort of the multiplicative group, and $u,v$ and $w$ for the  sort of the additive group. The function $\cdot$ and the constant $1$ are of the first sort and the function $+$ and the constant $0$ of the second sort:

\begin{equation}\label{multi}
\begin{array}{l}
\phi=\left\{\begin{array}{l}
\forall x\forall y\forall z((x\cdot y)\cdot z=x\cdot (y\cdot z))\\
\forall x(x\cdot 1=1\cdot x =x)\\
\forall x\forall y(x\cdot y= y\cdot x)\\
\forall x\exists y(x\cdot y =1)\end{array}\right.\\
\psi=\left\{\begin{array}{l}\forall x\forall y\forall z((x+ y)+ z=x+ (y+ z))\\
\forall x(x+ 0=0+ x=x)\\
\forall x\forall y(x+ y= y+ x)\\
\forall x\exists y(x+ y=0)\\
\forall x\forall y\forall z(x\cdot(y+z)=x\cdot y+x\cdot z)\\
\forall x\exists u(x=u)\wedge\forall u\exists x(u=0\vee u=x)
\end{array}\right.
\end{array}\end{equation}

\noindent We have separated the multiplicative group into  the first sort and the additive group  in the second sort. With this separation of the group and the bigger field part we can ask questions such as: $$\mbox{What kind of groups are the multiplicative group of a field?}$$ And the answer is: exactly the groups that satisfy 
\begin{equation}\label{sort}
\mbox{For some $+$ and for some $0$:\ }\ \phi\wedge\psi.
\end{equation}
The truth of the sentence (\ref{sort}) in a given group means that there is something out there outside the group, in this case the element $0$, which together with the new function $+$ defines a field.

For a different type of example, suppose 
\begin{equation}\label{exm1}
\phi
\end{equation}
is a finite second order axiomatization of some mathematical structure  in the vocabulary $\{R_1,...,R_n\}$. Suppose we want to say that $\phi$ has a model. So let us  take a new unary predicate $P$ and consider the sentence

\begin{equation}\label{exm2}
\exists P(\exists R_1\ldots\exists R_n\phi)^{(P)},
\end{equation}
where $\psi^{(P)}$ means the relativization of $\psi$ to the unary predicate $P$. What (\ref{exm2}) says in a model is that there are a subset $P$ and relations $R_1,...,R_n$ on $P$ such that $$(P,R_1,...,R_n)\models\phi.$$ So in any model which is big enough to include a model of $\phi$ the sentence   (\ref{exm2}) says that there indeed is such a model. But in smaller models (\ref{exm2}) is simply false, even though $\phi$ may have models. So (\ref{exm2}) does not really express the existence of a model for $\phi$. The situation would be different if we allowed $``\exists P\exists R_1\ldots\exists R_n"$ to refer to outside the model. In sort logic, which we will introduce in detail below, the meaning of the sentence

\begin{equation}\label{exm3}
\Exists R_1\ldots\Exists R_n\phi,
\end{equation}
is that there is a {\em new} domain of objects with new relations $R_1,...,R_n$ such that $\phi$ holds. Thus (\ref{exm3}) expresses the semantic consistency of $\phi$ independently of the model where it is considered.

In algebra concepts such as a module $P$ being projective, a group $F$ being free, etc,  are defined by reference to arbitrary modules $M$ and arbitrary groups $G$ with no concern as to whether such modules $M$ can be realized inside $P$, or whether such groups $G$ could be realized inside $F$. Even if it turned out that they could be so realized, the original concepts certainly referred to quite arbitrary objects $N$ and $G$ in the universe of all mathematical objects. Lesson: Apparently second order concepts in mathematics sometimes refer to outside the structure being considered.

Reference outside is, of course, most blatant in set theory where objects are defined by reference to the entire universe of sets. In practice one can in most cases limit the reference to some smaller part of the universe, but very often not to the elements or to the power-set of the object being defined.

\section{Sort logic}

Many-sorted logic has several domains, and variables for each domain, much like vector spaces have a scalar-domain and a vector-domain and different variables for each, or as geometry has different variables for points and lines. It seems to have been first considered by Herbrand, and later by Schmidt, 
Feferman \cite{MR0406772}, and others.


\subsection{Basic Concepts}

%
%

\def\sharp{\mathfrak{a}}
\def\srt{\mathfrak{s}}

A {\em (many-sorted) vocabulary\index{vocabulary}} is any set $L$ of
{predicate symbols} $P, Q, R, \ldots$. We leave function and constant symbols out for simplicity of presentation. We use natural numbers as names for sorts.

Each vocabulary $L$ has
an {\em arity-function}
$$\sharp_L : L \to \oN$$
which tells the arity of each predicate symbol, and a {\em sort-function} 
$$\srt_L:L\to\bigcup_n\oN^n, \srt_L(R)\in\oN^{\sharp_L(R)},$$
which tells what are the sorts of the elements of the tuple in a relation.
Thus if $P \in L$, then $P$
is an $ \sharp_L (P)$-ary predicate symbol for a relation of $\sharp_L(P)$-tuples of elements of sorts $n_1,\ldots,n_k$, where $(n_1,\ldots,n_k)=\srt(P)$.  So we can read off from every $n$-ary predicate symbol what the sorts of the elements are in the $n$-tuples of the intended relation. In other words, we do not have symbols for abstract relations between elements of arbitrary sorts (except identity $=$).

\subsection{Syntax}

The syntax of sort logic is very close to the syntax of second order logic. In effect we just add a new form of formula $\Exists P\phi$ with the intuitive meaning that there is a predicate $P$ of {\em new} sorts of elements such that $\phi$.

Suppose $L$ is a vocabulary.
Variable symbols for individuals are $x,y,z,...$ with indexes $x_0,x_1,...$ when necessary, and for relations  $X,Y,Z,...$ with indexes $X_0,X_1,...$. 
Each individual variable $x$ has a sort $\srt(x)\in\oN$ associated to it, so it is a variable for elements of sort $\srt(x)$. Each relation variable 
$X$ has an arity $\sharp(X)$ and a sort $\srt(X)\in\oN^{\sharp(R)}$ associated to it, so it is a relation variable for a relation between elements of the sorts $n_1\dot n_k$, where  $\srt(X)=(n_1\dot n_k)$.

The {\em logical symbols} of sort logic  of
the vocabulary $L$ are
$${\approx},\neg,\wedge,\vee,\forall,\exists,(,),x,y,z,\ldots,X,Y,Z,\ldots.$$
$L$-{\em equations} are of the form
$x=y$ where $x$ and $y$ can be variables of any sorts.
$L$-{\em atomic
formulas} are either $L$-{equations} or of the form $Rx_1 \ldots
x_k$, where $R \in L$, $\srt_L(R)=(n_1,\ldots,n_k)$, and $x_1, ..., x_k$ are
individual variables such that $\srt(x_i)=n_i$ for $i=1\dot k$. A {\em basic
formula} is an atomic formula or the
negation of an atomic formula. $L$-{\em formulas} are of the form:

\begin{enumerate}

\item 
$x=y$ 
\item 

$R(x_1 \dot x_n)$,\mbox{ when $\srt_L(R)=(\srt(x_1),...,\srt(x_n))$}

\item 
$X(x_1 \dot x_n)$,\mbox{ when $\srt(X)=(\srt(x_1),...,\srt(x_n))$}

\item 
$\neg \varphi$

\item 
$(\varphi \vee \psi)$

\item 
$\exists x  \varphi$. 

\item 
$\exists X \varphi$. 
\item 
$\Exists X \varphi$. {\bf New Sort Condition}: If  $\srt(X)=(n_1\dot n_k)$, then $\phi$ has no free variables or symbols of $L$, other than $X$, of a sort $n_i$ or of the sort $(m_1\dot m_l)$ with $\{m_1\dot m_l\}\cap\{n_1\dot n_k\}\ne\emptyset$. 
\end{enumerate}

The reason for the New Sort Condition is that the domains of the elements referred to by the free variables of $\phi$ are fixed already so they should not be altered by the  $\Exists$-quantifier.

We treat $\phi\wedge\psi$, $\phi\to\psi$, $\forall x\phi$ and $\Forall X\phi$ as shorthands obtained from disjunction and existential quantification by means of negation.

The concept of a free occurrence of a variable in a formula is defined as in first order logic. As a new concept we have the concept of a {\em free occurrence of a sort} in a formula. We define it as follows, following the intuition that if a sort  occurs ``free" in a formula, either as the sort of an individual variable, relation variable or predicate symbol, then to understand the meaning of the formula in a model we have to fix the domain of elements of that sort. Respectively, if a sort has only ``bound" occurrences in a formula, we can understand the meaning of the formula in a model without fixing the domain of elements of that sort, rather, while evaluating the meaning of the formula in a model we most likely try different domains of elements of  that sort.

\def\fs{\mathfrak{fs}}

The {\em free sorts} $\fs(\phi)$ of a formula are defined as follows:

\begin{enumerate}

\item 
$\fs(x=y)=\{ \srt(x),\srt(y)\}$
\item 

$\fs(Rx_1 \ldots x_n)=\{\srt(x_1),...,\srt(x_n)\}$

\item 
$\fs(Xx_1 \ldots x_n)=\{\srt(x_1),...,\srt(x_n)\}$

\item 
$\fs(\neg \varphi)=\fs(\phi)$

\item 
$\fs(\varphi \vee \psi)=\fs(\phi)\cup\fs(\psi)$

\item 
$\fs(\exists x  \varphi)=\fs(\phi)\cup\{\srt(x)\}$

\item 
$\fs(\exists X \varphi)=\fs(\phi)\cup\{n_1\dot n_k\}$, if $\srt(X)=(n_1\dot n_k)$. 

\item 
$\fs(\Exists X \varphi)=\fs(\phi)\setminus\{n_1\dot n_k\}$, if $\srt(X)=(n_1\dot n_k)$. 
\end{enumerate}

\subsection{Axioms}

Below $\phi(y/x)$ means the formula obtained from $\phi$ by replacing $x$ by $y$ in its free occurrencies. Substitution should respect sort.

\begin{definition}
The axioms of sort logic are as follows:

\noindent{\bf Logical axioms:}\begin{itemize}
\item Tautologies of propositional logic.
\item Identity axioms: {}{$x=y$, $x=y\to y=x$,  $(x_1=y_1\wedge...\wedge x_n=y_n\wedge\phi)\to\phi(y_1...y_n/x_1...x_n)$}, for atomic $\phi$
\item Quantifier axioms: 
	\begin{itemize}
	\item {}{$\phi(y/x)\to\exists x\phi$}, if $y$ is free for $x$ in $\phi$ in the usual sense. 
	\item {}{$\phi(Y/X)\to\exists X\phi$}, if $Y$ is free for $X$ in $\phi$ in the usual sense. 
	\item {}{$\phi(Y/X)\to\Exists X\phi$}, if $Y$ is free for $X$ in $\phi$ in the usual sense. 
	\end{itemize}
\end{itemize}

\noindent {\bf The rules of proof:}

\begin{itemize}
\item  Modus Ponens $\{\phi,\phi\to\psi\}\models\psi$
\item Generalization 
\begin{itemize}

\item $\{\Sigma,\phi\to\psi\}\models\exists x\phi\to\psi$, if $x$ is not free in $\Sigma\cup\{\psi\}$

\item $\{\Sigma,\phi\to\psi\}\models
	\exists X\phi\to
		\psi$, if $X$ is not free  in $\Sigma\cup\{\psi\}$
\item $\{\Sigma,\phi\to\psi\}\models
	\Exists X\phi\to
		\psi$, if  no free sorts of $\psi$ occur in $\srt(X)$.
\end{itemize}
\end{itemize}

\noindent{\bf First Comprehension Axiom:}
$$\exists X\forall y_1...\forall y_m(Xy_1...y_m\leftrightarrow\psi)$$ for any formula $\psi$ not containing $X$ free, whenever $\srt(X)=(\srt(y_1),\ldots,\srt(y_m)).$

\noindent{\bf Second Comprehension Axiom:}
$$\Exists X\forall y_1...\forall y_m(Xy_1...y_m\leftrightarrow\psi)$$ for any formula $\psi$ not containing $X$ free, whenever $\srt(X)=(\srt(y_1),\ldots,\srt(y_m)).$

\end{definition}

\def\U{\mathcal{U}}
 
The logical axioms and the rules of proof are clearly indispensable and are directly derived from corresponding axioms and rules of first order logic. The difference between the axioms $\phi(Y/X)\to\exists X\phi$ and $\phi(Y/X)\to\Exists X\phi$ is the following: Both take $\phi(Y/X)$ as a hypothesis. The conclusion $\exists X\phi$ says of the current sorts that a relation $X$ satisfying $\phi$ exists, namely $Y$. If $\srt(X)=(n_1\dot n_k)$, then the  conclusion $\Exists X\phi$ says of the sorts other than $n_1\dot n_k$ that domains for the sorts $n_1\dot n_k$ exists so that in the combined structure of the old and new domains a relation $X$ satisfying $\phi$ exists, namely $Y$.
 The Comprehension Axiom is the traditional (impredicative) axiom schema which gives second order logic, and in our case sort logic, the necessary power to do mathematics \cite{MR0351742}. In individual cases less comprehension may be sufficient but this is the general schema. The difference between the First and the Second Comprehension Axiom is that the former stipulates the existence of a relation $X$ defined by $\psi$ in the structure consisting of the existing sorts, while the latter says that this is even true if the sorts of elements and relations that $\psi$

If we limit ourselves to just one sort, for example $0$, we get exactly the classical second order logic.

\subsection{Semantics}

We now define the semantics of sort logic. This is very much like the semantics of second order logic, except that we have to take care of the new domains that may arise from interpreting quantifiers of the form $\Exists$ and $\Forall$.

\begin{definition}\label{stru defn}
An $L$-{\em structure\index{structure}} (or $L$-{\em
model\index{model}\index{model|see{also structure}}}) is a function $\mm$  defined on $L$ with the following properties:
\begin{enumerate}

\item If $R\in L$ and $\srt(R)=(n_1,...,n_k)$ then $n_i\in\dom(\mm)$ and $M_{n_i}=_{df}\mm(n_i)$ is a non-empty set for each $i\in\{1,...,k\}$. 
\item If $R \in L$ is an $k$-relation symbol and $\srt(R)=(n_1,...,n_k)$, then
$\mm (R) \subseteq M_{n_1}\times\ldots\times M_{n_k}$.

\end{enumerate}
\end{definition}

We usually shorten $\mm(R)$ to $R^\mm$. If no confusion arises, we use the notation 
$$\mm=(M_{n_1},\ldots,M_{n_l};R_1^\mm,\ldots,R_m^\mm)$$
for a many-sorted structure with universes $M_{n_1},\ldots,M_{n_l}$ and relations $R_1^\mm,\ldots,R_m^\mm$ between elements of some of the universes.
A vector space with scalar field $F$ and vector group $V$ would be denoted according to this convention (taking functions and constant relationally):
$$(V,F;\ \cdot\ ,1,\ +\ ,0).$$

\begin{definition}\label{semantics_fol}
An {\em assignment\index{assignment}} into an $L$-structure $\mm$ is any
function $s$ the domain of which is a set of individual variables, relation  variables and natural numbers such that 
\begin{enumerate}

\item 
If $x\in\dom(s)$, then $\srt(x)\in\dom(\mm)$ and $s(x)\in M_{\srt(x)}$. 
\item 
If $X\in\dom(s)$ with $\srt_L(X)=(n_1,\ldots,n_k)$, then $n_1,\ldots, n_k\in \dom(\mm)$ and $s(X)\subseteq M_{n_1}\times\ldots\times M_{n_k}$.

\end{enumerate}

\end{definition}



%
A {\em modified assignment} is defined as follows:

\begin{eqnarray*}
s[a/x](y) &=& \left\{
              \begin{array}{ll}
              a & \textrm{if $y=x$} \\
              s(y) & \textrm{otherwise}.
              \end{array}\right.\\          
    s[A/X](Y) &=& \left\{
              \begin{array}{ll}
              A & \textrm{if $Y=X$} \\
              s(Y) & \textrm{otherwise}.
              \end{array}\right.              
\end{eqnarray*}

Suppose $\srt(X)=(n_1\dot n_k)$. A model $\mm'$ is an {\em $X$-expansion} of a model $\mm$ if $\{n_1\dot n_k\}\cap\dom(\mm)=\emptyset$, $\dom(\mm')=\dom(\mm)\cup \{n_1\dot n_k\}$, and $\mm'\restriction\dom(\mm)=\mm$.

\begin{definition}
The {\em truth} of $L$-formulas in $\mm$ under $s$ is defined as
follows:

\begin{description}

\item[(1)] 
$\mm \models_s R(x_{1} \dot x_{n})$ if and only if  $(s(x_{1}), \ldots, s(x_{n})) \in\mm (R),$ 

\item[(2)] 
$\mm \models_s x=y$ if and only if  $s(x)=s(y)$,

\item[(3)] 
$\mm \models_s \neg \varphi$ if and only if  $\mm \nvDash_s
\varphi$,

\item[(4)] 

$\mm \models_s (\varphi \vee \psi)$ if and only if  $\mm \models_s
\varphi$ or $\mm \models_s \psi$

\item[(5)] 
$\mm \models_s \exists x \varphi$ if and only if 
 $\mm
\models_{s[a/x]} \varphi$ for some $a\in M_{\srt(x)}$, 

\item[(6)] 
$\mm \models_s \exists X\varphi$ if and only if  $\mm
\models_{s[A/X]} \varphi$ for some $A\subseteq M_{n_1}\times\ldots\times M_{n_k}$, 
where $\srt(X)=(n_1\dot n_k)$,
\item[(7)] 
$\mm \models_s \Exists X\varphi$ if and only if  $\mm'
\models_{s[A/X]} \varphi$ for some $X$-expansion  $\mm'$ of $\mm$ and some $A\subseteq M'_{n_1}\times\ldots\times M'_{n_k}$, 
where $\srt(X)=(n_1\dot n_k)$.

\end{description}
\end{definition}

Since (7)  of the above truth definition involves unbounded quantifiers over sets, the definition has to be given separately for formulas of quantifier-rank at most a fixed natural number $n$. When $n$ increases, the definition itself gets more complicated in the sense of the quantifier rank. 

As in second order logic, there is a looser concept of a model, one relative to which we can prove a Completeness Theorem. This concept permits also a uniform definition.


\begin{definition}\label{stru defn1}
A {\em Henkin} $L$-{\em structure\index{structure}} (or Henkin $L$-{\em
model\index{model}\index{model|see{also structure}}}) is a triple $(\mm,\U,\G)$, where $\mm$ is an $L$-structure, $\U$ is a set such that $\emptyset\notin\U$ and $\G$ is a set of relations between elements of the domains of $\mm$ and the sets in $\U$. We assume that the First and the Second Comprehension Axioms are satisfied by $(\mm,\U,\G)$ in the sense defined below.
\end{definition}

The idea is that $\U$ gives a set of possible domains for the new sorts needed for the truth conditions of the $\Exists$-quantifiers, and $\G$ gives a set of possible relations needed for the truth conditions of the $\exists$-quantifiers. Since $\U$ is not the class of all sets (as it is a set) and $\G$ need not be the set of {\em all} relevant relations, the structures $(\mm,\U,\G)$ are more general than the structures $\mm$. The original structures $\mm$ are called {\em full}.

An {\em assignment\index{assignment}} and a {\em modified assignment} for a Henkin $L$-structure $(\mm,\U,\G)$ is defined as for ordinary structures. Suppose $\srt(X)=(n_1\dot n_k)$. A model $\mm'$ is an {\em $X$-expansion in $\U$} of a model $\mm$ if $\{n_1\dot n_k\}\cap\dom(\mm)=\emptyset$, $\dom(\mm')=\dom(\mm)\cup \{n_1\dot n_k\}$, $\mm'\restriction\dom(\mm)=\mm$, and $\mm'(n_i)\in\U$ for all $i=1\dot k$.

\begin{definition}
The {\em truth} of $L$-formulas in $(\mm,\U,\G)$ under $s$ is defined as
follows:

\begin{description}

\item[(1)] 
$(\mm,\U,\G) \models_s R (x_{1} \dot x_{n})$ if and only if  $(s(x_{1}), \ldots, s(x_{n})) \in\mm (R)$, 

\item[(2)] 
$(\mm,\U,\G) \models_s x=y$ if and only if  $s(x)=s(y)$,

\item[(3)] 
$(\mm,\U,\G) \models_s \neg \varphi$ if and only if  $(\mm,\U,\G) \nvDash_s
\varphi$,

\item[(4)] 

$(\mm,\U,\G) \models_s (\varphi \vee \psi)$ if and only if  $(\mm,\U,\G) \models_s
\varphi$ or $(\mm,\U,\G) \models_s \psi$,

\item[(5)] 
$(\mm,\U,\G) \models_s \exists x \varphi$ if and only if 
 $(\mm,\U,\G)
\models_{s[a/x]} \varphi$ for some $a\in M_{\srt(x)}$, 

\item[(6)] 
$(\mm,\U,\G) \models_s \exists X\varphi$ if and only if  $(\mm,\U,\G)
\models_{s[A/X]} \varphi$ for some $A\in\P(M_{n_1}\times\ldots\times M_{n_k})\cap\G$, 
where $\srt(X)=(n_1\dot n_k)$,
\item[(7)] 
$(\mm,\U,\G) \models_s \Exists X\varphi$ if and only if  $(\mm,\U,\G)'
\models_{s[A/X]} \varphi$ for some $X$-expansion  $\mm'$ of $\mm$ in $\U$ and some $A\in\P(M'_{n_1}\times\ldots\times M'_{n_k})\cap\G$, 
where $\srt(X)=(n_1\dot n_k)$.

\end{description}
\end{definition}







The following characterization of provability in sort logic is proved as the corresponding result for type theory \cite{MR0036188}:

\begin{theorem}{\bf (Completeness Theorem)}
The following conditions are equivalent for any sentence $\phi$ of sort logic and any countable theory $T$ of sort logic: \begin{enumerate}
\item $T\models\phi$.
\item Every Henkin model of $T$ satisfies $\phi$.
\item Every countable Henkin model of $T$ satisfies $\phi$.
\end{enumerate} 
\end{theorem}

This characterization shows that our axioms for sort logic capture the intuition of sort logic in a perfect manner, at least if our Henkin semantics does. Out Henkin semantics is very much like that of second order logic.

\section{Sort logic and set theory}

In this chapter we look at sort logic from the point of vies of set theory.



\begin{definition}
We use $\Delta_n$ to denote the set of formulas of sort logic which are (semantically) equivalent both to a $\Sigma_n$-formula of sort logic, and to a $\Pi_n$-formula of sort logic. 
\end{definition}

\begin{theorem}\cite{MR567682}\cite{MR515154}
The following conditions are equivalent for any model class $K$ and for any $n>1$:
\begin{description}

\item[(1)] $K$ is definable in the logic $\Delta_n$.
\item[(2)] $K$ is $\Delta_n$-definable in the Levy-hierarchy.

\end{description}
\end{theorem}

\begin{proof} We give the proof only  in the case $n=2$. The general case is similar. Suppose  $L$ is a finite vocabulary and $\ma$ is a \soc\  $L$-structure. Suppose $\sigma$ is the conjunction of a large finite part of ZFC. Let us call a model $(M,\in)$ of $\theta$ {\em  supertransitive} if for every $a\in M$ every element and every subset of $a$ is in $M$. Let $\Strt(M)$ be a $\Pi_1$-formula which says that $M$ is supertransitive. Let $\Voc(x)$ be the standard definition of ``x is a vocabulary". Let $\So(L,x)$ be the set-theoretical definition of the class of second order $L$-formulas. Let $\Str(L,x)$ be the set-theoretical definition of $L$-structures. Let $\Sat(\ma,\phi)$ be an inductive truth-definition of the $\Sigma_2$-fragment of sort logic written in the language of set theory. Let 
\begin{eqnarray*}
P(z,x,y)&=&\Voc(z)\wedge \Str(z,x)\wedge \So(z,y)\wedge\\
&&\exists M(z,x,y\in M\wedge\sigma^{(M)}\wedge \Strt(M)\wedge (\Sat(z,x,y))^{(M)})
\end{eqnarray*}
Now  if $L$ is a vocabulary, $\ma$ an $L$-structure, then  $\ma\models\phi\iff P(L,\ma,\phi)$. This shows that $\ma\models\phi$ is a $\Sigma_2$ property of $\ma$ and $L$.

For the converse, suppose the predicate $\Phi$ is  a $\Sigma_2$ property of $L$-structures. There is a $\Sigma_2$-sentence $\phi$ of sort logic such that for all $\mm$, $\mm\in K$ off $\mm\models\phi$.

Suppose $\Phi=\exists x\forall yP(x,y,\ma)$, a $\Sigma_2$-property of $\ma$. Let $\psi$ be a sort logic sentence the models of which are, up to isomorphism, exactly the models $\ma$ for which there is 
$(V_\alpha,\in)$, with $\alpha=\beth_\alpha$, $\ma\in V_\alpha$, and 
$(V_\alpha,\in)\models\exists x\forall yP(x,y,\ma)$. 
If $\exists x\forall yP(x,y,\ma)$ holds, we can 
find a model for $\psi$ by means of the Levy Reflection principle.
On the other hand, suppose $\psi$ has a model $\ma$. W.l.o.g. it is of the form $(V_\alpha,\in)$ with $\ma\in V_\alpha$. Let $a\in V_\alpha$ such that $(V_\alpha,\in)\models\forall y P(a,y,\ma)$. Since in this case $H_\alpha=V_\alpha$,  $(H_\alpha,\in)\models\forall y P(a,y,\ma)$, where $H_\alpha$ is the set of sets of hereditary cardinality $<\alpha$. By another application of the Levy Reflection Principle we get $(V,\in)\models\forall y P(a,y,\ma)$, and we have proved $\exists x\forall yP(x,y,\ma)$.

\end{proof}

By a {\em model class} we mean a class of structures of the same vocabulary, which is closed under isomorphisms. In the context of set theory classes are referred to by their set-theoretical definitions. 

The following consequence was mentioned in \cite{MR0457146} without proof:
\begin{corollary}\cite{MR567682}\cite{MR515154}
Every model class is definable in sort logic. Sort logic is therefore the strongest logic.
\end{corollary}

The logics $\Delta_n$, $n=2,3,\ldots$, provide a sequence of stronger and stronger logics. Their model theoretic properties can be characterized in set theoretical terms as the following results indicate:

\begin{theorem}\cite{MR567682}
The Hanf-number of the logic $\Delta_n$ is $\delta_n$. The L\"owenheim number of the logic $\Delta_n$ is $\sigma_n$. The decision problem of the logic $\Delta_n$ is the complete $\Pi_n$-definable set of natural numbers.\end{theorem}

\begin{theorem}\cite{MR0295904} The LST-number of $\Delta_2$ is the first supercompact cardinal.

\end{theorem}

\begin{theorem} The LST-number of $\Delta_3$ is at least the first extendible cardinal.

\end{theorem}

\begin{theorem}
The decision problem of $\Delta_n$ is the complete $\Pi_n$-set of natural numbers.
\end{theorem}

\section{Sort logic and foundations of mathematics}

We suggest that sort logic can provide a foundation of mathematics in the same way as second order logic, with the strong improvement that it does not depend on the ad hoc Large Domain Assumptions of set theory.

We will now think of foundations of mathematics from the point of view of sort logic. Propositions of mathematics are---according to the sort logic view---either of the form $$\ma\models\phi,$$ where $\ma$ is a structure, characterizable in sort logic, and $\phi$ is a sentence of sort logic, or else of the form $$\models\phi,$$ where again $\phi$ is a sentence of sort logic. Thus a proposition of mathematics either states a {\em specific} truth, truth in a specific structure, or a {\em general} truth, truth in all structures. The specific truth can be reduced to the general truth as follows. Suppose $\theta_\ma$ is a sort logic sentence which characterizes $\ma$ up to isomorphism. Then $$\ma\models\phi\iff\models\theta_\ma\to\phi.$$

Curiously, and quite unlike the case of second order logic, the converse holds, too. Suppose $\phi$ is a sort logic sentence in which the predicates $P_1,...,P_k$ occur only. Let $X_1\dot X_k$ be new unary predicate variables  of sorts $\srt(P_1),...,\srt(P_k)$ respectively.  Then $$\models\phi\iff\ma\models\Forall  X_1\ldots\Forall X_k\phi(X_1\ldots X_k/P_1\ldots P_k)$$
$$\not\models\phi\iff\ma\models\neg\Forall  X_1\ldots\Forall X_k\phi(X_1\ldots X_k/P_1\ldots P_k).$$  Intuitively this says that a sort logic sentence which talks about the predicates $P_1\dot P_k$ in some domains is valid if and only if whatever new domains and interpretations we take for $P_1\dot P_k$, $\phi$ is true. Since the general truth is reducible to specific truth we may focus on specific truth only, without loss of generality.

What is the justification we can give to asserting $\ma\models\phi$? We can {\em prove} from the axioms of sort logic the sentence $\theta_\ma\to\phi$. Of course, we may have to go beyond the standard axioms of sort logic, but much of mathematics can be justified with the sort logic axioms that we have.

What is the justification for asserting that $\phi$ has a model, i.e. $\neg\phi$ is not valid? By the above it suffices to prove from the axioms the sentence $$\neg\Forall  X_1\ldots\Forall X_k\phi(X_1\ldots X_k/P_1\ldots P_k).$$ If we compare the situation of sort logic with that of second order logic the difference is that in second order logic we have to make so-called ``large domain assumptions" to justify existence of mathematical structures, while in sort logic we can simply prove them from the general Comprehension Axioms. But here comes a moment of truth. Can we actually prove the existence of the structures necessary in mathematics from the mere Comprehension Axioms?

In fact we need more axioms to supplement the First and the Second Comprehension Axiom.

\begin{definition}
{\bf Power Sort Axiom:} 
\begin{equation}\label{psa}
\begin{array}{l}
\Exists Y(\forall u\exists z_1(u=z_1)\wedge\forall z_1\exists u(u=z_1)\wedge\\
 \forall x\forall y
(\forall z_1\ldots\forall z_n(Yxz_1\ldots z_n\leftrightarrow Yyz_1\ldots z_n)\to x=y)\wedge\\
\forall X\exists x\forall z_1\ldots \forall z_n(Xz_1\ldots z_n\leftrightarrow Yxz_1\ldots z_n)),
\end{array}
\end{equation}
where $$\begin{array}{l}
\srt(X)=(\srt(z_1),\ldots,\srt(z_n))\\
\srt(Y)=(\srt(x),\srt(z_1),\ldots,\srt(z_n))\\
\srt(x)=\srt(y)\\
\srt(z_1)=\ldots=\srt(z_n).
\end{array}$$
\end{definition}

 Note that in the Power Sort Axiom only $\srt(u)$ occurs free, so it is an axiom about models with one sort, namely $\srt(u)$. Naturally the models may consist of other sorts as well, this axiom just does not say anything about those sorts. The sort $\srt(z_1)$ is just an auxiliary copy of the sort $\srt(u)$, as the conjunct $\forall u\exists z_1(u=z_1)\wedge\forall z_1\exists u(u=z_1)$ stipulates. The sort $\srt(x)$ is a new sort which codes the $n$-ary relations on the domain $\srt(u)$. The coding is done by means of the predicate $Y$.

\begin{lemma}
Every full model satisfies the Power Sort Axiom.
\end{lemma}

\begin{proof}
Suppose $\mm$ is a (full) model and $s$ is an assignment into $\mm$. Let us fix $X$, $Y$, and $x,y,z_1,\ldots, z_n$ such that $$\srt(X)=(\srt(z_1),\ldots,\srt(z_n)), \srt(Y)=(\srt(x),\srt(z_1),\ldots,\srt(z_n)),$$ $\srt(x)=\srt(y)$ and $\srt(z_1)=\ldots=\srt(z_n)$. Let $\mn$ be like $\mm$ except that there is a new sort $\srt(u)$ (or if this sort existed in $\mm$ it is now replaced) with universe $\P(M_{\srt(x)})$ and $\mn_{\srt(z_1)}=M_{\srt(u)}$. Let $s'$ be like $s$ except that $$s'(Y)=\{(a,b_1,\ldots,b_n)\in N_{\srt(x)}\times M_{\srt(u)}
\times\ldots\times M_{\srt(u)}
 : (b_1,\ldots,b_n)\in a\}.$$ Now $s'$ satisfies in $\mn$ the formula
$$
\begin{array}{l}
 \forall x\forall y
(\forall z_1\ldots\forall z_n(Yxz_1\ldots z_n\leftrightarrow Yyz_1\ldots z_n))\to x=y)\wedge\\
\forall X\exists x\forall z_1\ldots \forall z_n(Xz_1\ldots z_n\leftrightarrow Yxz_1\ldots z_n)),
\end{array}$$ Hence $s$ satisfies in $\mm$ the formula~(\ref{psa}).

\end{proof}

\begin{definition}
{\bf Infinite Sort Axiom:}
\begin{equation}\label{isa}
\begin{array}{l}
\Exists X(\forall x\forall y\forall z
((Xxy\wedge Xxz)\to y=z)\\
\forall x\forall y\forall z((Xxz\wedge Xyz)\to x=y)\\
\forall x\exists yXxy\\
\exists z\forall x\forall y(Xxy\to \neg y=z))
\end{array}
\end{equation}
where $\srt(X)=(\srt(x),\srt(x))$ and $\srt(x)=\srt(y)=\srt(z)$.
\end{definition}

\begin{lemma}
Every full model satisfies the Infinite Sort Axiom.
\end{lemma}

\begin{proof}
Suppose $\mm$ is a (full) model and $s$ is an assignment into $\mm$. Let us fix $X$,  and $x,y,z$ such that $\srt(X)=(\srt(x),\srt(x))$, and $\srt(x)=\srt(y)=\srt(z)$. Let $\mm'$ be like $\mm$ except that there is a new sort $\srt(x)$ (or if this sort existed in $\mm$ it is now replaced) with universe $\omega$. Let $s'$ be like $s$ except that $$s'(X)=\{(n,n+1)
 : n\in\omega\}.$$ Now $s'$ satisfies in $\mm'$ the formula
$$
\begin{array}{l}
\forall x\forall y\forall z
((Xxy\wedge Xxz)\to y=z)\\
\forall x\forall y\forall z((Xxz\wedge Xyz)\to x=y)\\
\forall x\exists yXxy\\
\exists z\forall x\forall y(Xxy\to \neg y=z)
\end{array}.$$
 Hence $s$ satisfies in $\mm$ the formula~(\ref{isa}).

\end{proof}

The Power Sort Axiom, reminiscent of the Power Set Axiom of set theory, is necessary for arguing about the existence of new sorts of elements. The Infinite Sort Axiom is required for arguing about infinite domains, just as we need the Axiom of Infinity in set theory. Note that  the Power Sort Axiom or the Infinite Sort Axiom do not imply that we have only infinite or uncountable models. By the above lemmas  these axioms  are true in all models, even finite ones.

With the above two new axioms we can construct mathematical structures up to any cardinality $<\beth_\omega$, as if we were working in Zermelo's set theory. For bigger structures we have to make stronger assumptions, and they probably have great similarity with the Replacement Axiom of set theory. The point is that in second (and higher) order logic we have to make ad hoc Large Domain Assumptions as we go from structure to structure, while in sort logic we need only make general assumptions about domains, as axioms of set theory postulate general properties of sets. 

So what is the difference between sort logic and set theory?  Despite its proximity to set theory, sort logic is still a logic, like first order logic, second order logic, infinitary logic, etc. Sort logic treats mathematical structures up to isomorphism only, there is no preference of one construction of a structure over another, and this is in line with common mathematical thinking about structures. In its model theoretic formulation sort logic gives rise to interesting reflection principles via its L\"owenheim-Skolem-Tarski properties. Finally, sort logic provides a natural model theoretic forum for investigating complicated set theoretical properties
of models, without going into the nuts and bolts  of the constructions of specific structures. 

As the strongest logic sort logic is an ultimate yard-stick of definability in mathematics. Any property that is isomorphism invariant can be measured by sort logic. The canonical hierarchy $\Delta_n$ ($n<\omega$) inside sort logic climbs up the large cardinal hierarchy by reference to Hanf-, L\"owenheim- and Skolem-L\"owenhem-Tarski-numbers, reaching all the way to Vopenka's Principle. The  model classes that are $\Delta_2$ in the Levy-hierarchy are exactly the model classes definable in the $\Delta$-extension of second order logic. Sort logic provides a similar characterization of model classes that are $\Delta_n$ in the Levy-hierarchy for $n>2$. Is this too strong a logic to be useful? For logics as strong as sort logic the main use is in definability theory. But sort logic has also a natural axiomatization, complete with respect to a natural concept of a Henkin model, so we can also write inferences in sort logic. This is an alternative way of looking at mathematics to set theory, one in which definition rather than construction is the focus.







\end{document}